\documentclass[preprint,1p]{elsarticle}

\makeatletter
 \def\ps@pprintTitle{%
 	\let\@oddhead\@empty
 	\let\@evenhead\@empty
 	\def\@oddfoot{\footnotesize\itshape
 		{} \hfill\today}%
 	\let\@evenfoot\@oddfoot
 }
\makeatother
\usepackage[unicode]{hyperref}

\usepackage{latexsym}
\usepackage{indentfirst}
\usepackage{amsxtra}
\usepackage{amssymb}
\usepackage{amsthm}
\usepackage{amsmath}
\usepackage{mathrsfs} 

\usepackage{xcolor}
\usepackage{color}

\usepackage{amsfonts}
\usepackage{mathtools}

\usepackage[capitalise]{cleveref}

\newtheorem{theor}{Theorem}[section]
\newtheorem{prop}[theor]{Proposition}
\newtheorem{cor}[theor]{Corollary}
\newtheorem{lemma}[theor]{Lemma}
\theoremstyle{definition} 
\newtheorem{defin}[theor]{Definition}

\newtheorem{ex}{Example}
\newtheorem{exs}{Examples}

\newtheorem{conv}{Convention}

\DeclareMathOperator{\Sym}{Sym}
\DeclareMathOperator{\id}{id}

\begin{document}

\begin{frontmatter}
	\title{Left seminear-rings, groups semidirect products and left cancellative left semi-braces 
}
	\author[unile]{M.~CASTELLI	\tnoteref{mytitlenote}}
	\tnotetext[mytitlenote]{The author is member of GNSAGA (INdAM). ORCID: 0000-0002-6055-1328}
	\ead{marco.castelli@unisalento.it - marcolmc88@gmail.com}

\begin{abstract}
We study some relations between left cancellative left semi-braces and other existing algebraic structures. In particular, we show that every left semi-brace arises from a left seminear-ring, extending the correspondence given by Rump between skew left braces and left near-rings in \cite{rump2019set}. Moreover, we show a correspondence between certain groups semidirect products and left cancellative left semi-braces satisfying an additional hypothesis on the set of idempotents. As an application, we classify left cancellative left semi-braces of size $pq$ and $2p^2$ such that the set of idempotents $E$ is a Sylow subgroup of the multiplicative group. Finally, we study various type of nilpotency, recently introduced in \cite{catino2022nilpotency}, of these left semi-braces. 
\end{abstract}

\begin{keyword}
\texttt{set-theoretic solution\sep Yang-Baxter equation\sep semi-brace \sep brace \sep skew brace}
\MSC[2020] 
16T25\sep 81R50 
\end{keyword}
\end{frontmatter}

\section*{Introduction}

The Yang-Baxter equation, that comes from the papers due by Yang \cite{yang1967} and Baxter \cite{baxter1972partition}, is a basic equation of statistical mechanics. Classifying the solutions of the Yang-Baxter equation seems to be hard and it is still an open problem. In 1992 Drinfel'd \cite{drinfeld1992some} moved the interest of several researchers for finding and classifying a particular class of solutions of this equation, the so-called set-theoretic solutions. Namely, a \emph{set-theoretic solution of the Yang-Baxter equation}, or simply a \emph{solution}, on a non-empty set $X$ is a pair $\left(X,r\right)$, where 
$r:X\times X\to X\times X$ is a map such that the relation
\begin{align*}
\left(r\times\id_X\right)
\left(\id_X\times r\right)
\left(r\times\id_X\right)
= 
\left(\id_X\times r\right)
\left(r\times\id_X\right)
\left(\id_X\times r\right)
\end{align*}
is satisfied.  
Writing a solution $(X,r)$ as $r\left(x,y\right) = \left(\lambda_x\left(y\right)\rho_y\left(x\right)\right)$, with
$\lambda_x, \rho_x$ maps from $X$ into itself, for every $x\in X$, we say that $(X, r)$ is \emph{left non-degenerate} if $\lambda_x\in \Sym_X$ for every $x\in X$, \emph{non-degenerate} if $\lambda_x,\rho_x \in \Sym_X$ for every $x\in X$, and \emph{involutive} if $r^2=\id_{X\times X}$. 
After the seminal papers of Gateva-Ivanova and Van den Bergh \cite{gateva1998semigroups}, Gateva-Ivanova and Majid \cite{gateva2008matched}, and Etingov, Schedler, and Soloviev \cite{etingof1998set}, the involutive non-degenerate solutions have been widely studied.\\
In order to find and classify these solutions, several algebraic structures were provided. In 2007 Rump \cite{rump2007braces} introduced \emph{left braces}, algebraic structures that include Jacobson radical rings. Left braces determine set-theoretic solutions, and conversely, given a set-theoretic solution $(X,r)$, one can associate a left brace $\mathcal{G}(X,r)$. As showed in several recent papers, left braces can be used to take informations on set-theoretic solutions such as retractability, indecomposability and simplicity (see for example \cite{CaCaSt20x,cedo2014braces,cedo2022new,rump2020one} and related references).\\
Recently, various generalizations of left braces were introduced to study set-theoretic solutions that are not necessarily involutive. In \cite{guarnieri2017skew}, Guarnieri and Vendramin introduced skew left braces to study bijective non-degenerate solutions. As one can expect, skew left braces give structural informations on bijective non-degenerate solutions (see, for example, \cite{acri2020retractability,CeSmVe19,guarnieri2017skew,smock2018skew}). Skew left braces were furtherly generalized  by Catino, Colazzo and Stefanelli in \cite{cacs3} to deal with left non-degenerate solutions, not necessarily bijective. Technically, they introduced the notion of \emph{left cancellative left semi-brace}, that is a triple $(B,+,\circ)$ such that $(B,+)$ is a left cancellative semigroup, $(B,\circ)$ is a group and the equality
\begin{equation}\label{equalsemi1}
    a\circ (b+c)=a\circ b+a\circ (a^-+c)
\end{equation}
holds for all $a,b,c\in B$ (where $a^-$ is the inverse of $a$ in $(B,\circ)$). In \cite{cacs3} these structures are simply called left semi-braces: here, we will use the prefix \emph{left cancellative} to distinguish them from the left semi-braces studied in \cite{JeVa19}, where the left cancellativity of $(B,+)$ is not required. If $(B,+)$ is a group, then $(B,+,\circ)$ is a skew left brace, and if in addition $(B,+)$ is abelian then $(B,+,\circ)$ is a left brace. For further generalizations of left braces, we refer the reader to \cite{brzezinski2017towards,catino2021set,catino2022set,catino2021inverse,colazzo2021algebraic} and related references. \\
By \cite[Theorem 9]{cacs3}, a left cancellative left semi-braces $(B,+,\circ)$ gives a left non-degenerate solution by $\lambda_a(b):=a\circ(a^-+b)$ and $\rho_b(a):=(a^-+b)^-\circ b$ for all $a,b\in B$. On the other hand, skew left braces and their generalizations have many connections with other topics such as Hopf-Galois extensions, regular subgroups, rings and near rings, left-orderable groups: for this reason, they increased of interest indipendently of the Yang-Baxter equation.\\ 
Several authors studied and classified large classes of skew left braces and investigated the ones having specific properties (see for example \cite{acri2020skew,bachiller2015classification,dietzel2021braces,rump2019construction,smoktunowicz2018engel,stefanello2022bi}).
At the same time, some other papers studied general effects of the "distributive law" to left braces and their generalizations. To this end, several unifying theories were developed using ring and group theoretical tools. In this context, some remarkable links between skew left braces and other existing algebraic structures, such as rings, heaps and near-rings, were recently discovered (see for example \cite{brzezinski2017towards,brzezinski2019trusses,brzezinski2022pre,rump2019set}).\\
In this paper, we follow these two tracks to investigate left cancellative left semi-braces, that represent, up to now, a quite unexplored field. Indeed, after the seminal papers of Catino, Colazzo and Stefanelli \cite{cacs3} and Jespers and Van Antwerpen \cite{JeVa19} , the only interesting results on left semi-braces were given in \cite{catino2022nilpotency} and in \cite{stefanelli2022affine}. In the first one, Catino, Ced\'o and Stefanelli introduced suitable concepts of nilpotency for left cancellative left semi-braces, extending the ones previously given for skew left braces; in the second one, Stefanelli extended the concept of affine structure on a group given by Rump in \cite{rump2019construction} for left braces to study left cancellative left semi-braces $(B,+,\circ)$ having a fixed group $H$ as $(B,\circ)$.\\

The first aim to this paper is to extend the link between skew left  braces and left near-rings given by Rump in \cite[Theorem 1]{rump2019set} to left cancellative left semi-brace. As skew left braces come from suitable left near-ring, we will show that left cancellative left semi-braces come from suitable \emph{left seminear-rings}, algebraic structures already present in literature (see \cite{samman1998topics} for an exhaustive introduction and further references on this topic) that include rings and left near-rings. As one can expect, our correspondence is consistent with \cite[Theorem 1]{rump2019set} in the case of skew left braces.\\
The second goal of the paper is to give a kind of counterpart to \cite[Section 3]{catino2022nilpotency}, where left cancellative left semi-braces having the set of idempotents as an ideal were studied, considering left cancellative left semi-braces $(B,+,\circ)$ for which the skew left brace $G$ contained in $B$ is the ideal given by $G:=Ker(\lambda_{|_E})$, i.e. $G$ coincides with the kernel of the action on the set of additive idempotents induced by the map $\lambda$. We will show that these left semi-braces come from a skew left brace $G$ and a groups semidirect product $G\rtimes_{\beta} H$ where $\beta$ acts on the skew left brace $G$. As a main application, we will use these left semi-braces to give a complete classification of left cancellative left semi-braces having size $pq$ (for arbitrary prime numbers $p,q$) and $2p^2$ (for an arbitrary prime number $p>2$) such that the set of additive idempotents $E$ is a Sylow subgroup of $(B,\circ)$.\\
In the final section, we study left and right nilpotency, recently introduced in \cite{catino2022nilpotency}, of the left semi-braces provided in this paper. In particular, we will show that left cancellative left semi-braces with $G:=Ker(\lambda_{|_E})$ are almost never right nilpotent. 
On the other hand, a large larte of these left semi-braces are right nil: in this way, we answer in negative to \cite[Question 1]{catino2022nilpotency}. Finally, we consider left nilpotency and we provide a family of left nilpotent left semi-braces that is different from the one provided in \cite[Section 5]{catino2022nilpotency}

\section{Basic definitions and results}

In this section, we introduce the basic definitions and results involving left semi-braces mainly contained in \cite{catino2022nilpotency,cacs3,JeVa19}.

\begin{defin}
A set $B$ with two binary operation $+$ and $\circ$ is said to be a \emph{left semi-brace} is $(B,+)$ is a semigroup, $(B,\circ)$ is a group and the equality
\begin{equation}\label{equalsemi}
    a\circ (b+c)=a\circ b+a\circ (a^-+c)
\end{equation}
holds for all $a,b,c\in B$, where $a^-$ is the inverse of $a$ in $(B,\circ)$ (here, we assume that the operation $\circ$ has higher precedence than the operation $+$).  A left semi-brace is said to be \emph{left cancellative} if $(B,+)$ is a left cancellative semigroup.
\end{defin}

\noindent The neutral element of $\circ$ will be denoted by $0$. The simpler example of left semi-brace is the one given by an arbitrary group $(E,\circ)$ together with the $+$ operation given by $a+b:=b$ for all $a,b\in E$. Following the terminology of \cite{cacs3}, we refer to this left semi-brace as the \emph{trivial left semi-brace} on the group $(E,\circ)$. We remark that trivial left semi-braces are different from the \emph{trivial skew left braces}: indeed, a trivial skew left brace on a group $(B,\circ)$ will be given by $a+b:=a\circ b$ for all $a,b\in B$.\\
Trivial left semi-braces clearly are left cancellative: through this paper, all the left semi-braces will be left cancellative. In this case, it is well-known that $0$ is an idempotent respect to the operation $+$ and hence it is a left identity for $(B,+)$.  Moreover, the additive structure has been completely determined, as we can see in the following result contained in \cite{cacs3}.

\begin{prop}[Section $2$ and Proposition $5$, \cite{cacs3}]\label{addstrut}
Let $B$ be a left cancellative left semi-brace. Then, the additive semigroup $(B,+)$ is the direct sum $G+E$, where $G:=B+0$ is a subgroup of $(B,\circ)$ and $E$ is the set of idempotents of $(B,+)$, which also is a subgroup of $(B,\circ)$. Moreover, $(G,+,\circ)$ is a skew left brace and $(E,+,\circ)$ is a trivial left semi-brace.
\end{prop}

\begin{conv}
 From now on, given a left cancellative left semi-brace $(B,+,\circ)$, we will indicate by $G$ the skew left brace contained in $B$ and by $E$ the trivial left semi-brace given by the additive idempotents.
\end{conv}

\smallskip

\noindent For our scopes, the map $\lambda$, introduced in \cite{cedo2014braces} for left braces and generalized in \cite{cacs3,JeVa19} for left semi-braces, is of crucial importance.

\begin{prop}[Propositions 3-7, \cite{cacs3}]\label{lambdasemi}
Let $B$ be a left cancellative left semi-brace. Then, the map $\lambda_a$ given by $\lambda_a(b):=a\circ (a^-+b)$ is an automorphism of $(B,+)$ and give rise to a homomorphism $\lambda$ from $(B,\circ)$ to $Aut(B,+)$ by $\lambda(a):=\lambda_a$ for all $a\in B$. Moreover, the following properties hold:
\begin{itemize}
    \item[-] $\lambda_a(E)=E$ for all $a\in B$;
    \item[-] $\lambda_a(0)=0$ if and only if $a\in G$.
\end{itemize}
\end{prop}

\noindent We remark that the previous proposition give us an important information: the $\lambda$-action of $B$ induces an action of $(B,\circ)$ on $(E,+)$, i.e. a homomorphism from $(B,\circ)$ to $Aut(E,+)$, which is isomorphic to $Sym(|E|)$.\\
As for left braces and skew left braces, the notion of ideal of a left cancellative left semi-brace, which we recall below, is strictly related to the map $\lambda$.

\begin{defin}[Definition $17$, \cite{cacs3}]
Let $(B,+,\circ)$ be a left cancellative left semi-brace and $I$ a subset of $B$. Then, $I$ is an \emph{ideal} of $B$ if the following conditions hold:
\begin{itemize}
    \item[1.] $I$ is a normal subgroup of $(B,\circ)$;
    \item[2.] $I\cap G$ is a normal subgroup of $(G,+)$;
    \item[3.] $(n^-+b)^-\circ b\in I$ for all $b\in B$ and $n\in I\cap G$;
    \item[4.] $\lambda_g(e)\in I$ for all $g\in G$ and $e\in I\cap E$.

\end{itemize}

\end{defin}

\noindent By the previous definition, the set of idempotents $E$ is an ideal if and only if it is a normal subgoup of $(B,\circ)$. Moreover, by \cite[Proposition 7]{cacs3} it follows that $G$ also is an ideal if and only if it is a normal subgroup of $(B,\circ)$.

\smallskip 

\noindent In order to construct new examples of left cancellative left semi-braces, in \cite{cacs3} a suitable version of the semidirect product for left cancellative left semi-braces was provided.

\begin{prop}
Let $B_1,B_2$ be left cancellative left semi-braces and $\alpha$ a homomorphism from $B_2$ into the group of automorphisms of the left cancellative left semi-brace $B_1$. Then, the binary operations given by 
$$(b_1,c_1)+(b_2,c_2):=(b_1+b_2, c_1+c_2) $$
$$(b_1,c_1)\circ (b_2,c_2):=(b_1\circ \alpha_{c_1}(b_2),c_1\circ c_2) $$
for all $(b_1,c_1),(b_2,c_2)\in B_1\times B_2$ give rise to a left cancellative left semi-brace which we call the semidirect product of $B_1$ and $B_2$ via $\alpha$ and we denote it by $B_1\rtimes_{\alpha} B_2$.
\end{prop}

\noindent If $E$ is an ideal of $B$, then the whole left cancellative left semi-brace can be described in terms of a particular semidirect product, as showed in \cite{catino2022nilpotency}.

\begin{prop}[Theorem 16, \cite{catino2022nilpotency}]\label{semidi}
Let $B$ be a left cancellative left semi-brace. Then, $E$ is an ideal of $B$ if and only if $B$ is isomorphic to a semidirect product of left cancellative left semi-braces $T\rtimes A$ of a trivial left semi-brace $T$ by a skew left brace $A$.
\end{prop}

\noindent As a simple consequences of the previous proposition, we obtain the following results.

\begin{cor}\label{consesemi}
Let $B$ be a left cancellative left semi-brace such that $(B,\circ)$ is an abelian group. Then, $B$ is isomorphic to the direct product of the trivial left semi-braces $E$ and the skew left brace $G$. 
\end{cor}

\begin{proof}
It follows by \cref{semidi}.
\end{proof}

\begin{cor}\label{consesemi2}
Let $B$ be a left cancellative left semi-brace such that $(B,\circ)$ is a cyclic group having prime-power size or infinite size. Then, $B$ is a skew left brace or $B$ is a trivial left semi-brace.
\end{cor}

\begin{proof}
If $B$ is a left cancellative left semi-brace such that $(B,\circ)$ is a cyclic group having prime-power size, by \cref{consesemi} $B$ is the direct product of the left semi-braces $E$ and $G$, therefore $|E|=1$ or $|G|=1$, hence the thesis. The thesis similarly follows if $(B,\circ)$ is an infinite cyclic group.
\end{proof}

\medskip 

\noindent As we said before, the additive structure of a left cancellative left semi-brace has been completely determined. This fact has an important consequence on the multiplicative structure.

\begin{prop}\label{fatt}
Let $B$ be a left cancellative left semi-brace. Then, $B=G\circ E$ and, if $B$ is finite, then $|B|=|G|\cdotp |E|$.
\end{prop}

\begin{proof}
Let $b\in B$. Then, by \cref{addstrut} there exist $g\in G$ and $e\in E$ such that $b=g+e.$ Therefore $b=g+e=g\circ g^-\circ (g+e)=g\circ \lambda_{g^-}(e)$, where $\lambda_{g^-}(e)\in E $ by \cref{lambdasemi}. The last part follows since $(B,+)$ is the direct sum of $(G,+)$ and $(E,+)$.
\end{proof}

\section{Left seminear-rings and left semi-braces}

In this section, we show the correspondence between left cancellative left semi-braces and left seminear-rings, which extends the one given by Rump in \cite{rump2019set} between skew left braces and left near-rings.

\medskip

We start recalling basic definitions and results on left seminear-rings.

\begin{defin}
    A triple $(S,+,\cdotp)$ is said to be a \emph{left seminear-ring} if $(S,+)$ and $(S,\cdotp)$ are semigroups and $a\cdotp (b+c)=a\cdotp b+a\cdotp c$ for all $a,b,c\in S$. Moreover, a seminear ring is said to be \emph{unital} if $(S,\cdotp)$ is a monoid: in this case, the neutral element will be indicated by $1$.
\end{defin}

\noindent An element $o$ of a left seminear-ring $(S,+,\cdotp)$ is said to be a \emph{right zero} if $s\cdotp o=o$ for all $s\in S$.

\begin{exs}\label{essemi}
\begin{itemize}
    \item[1)] Every left near-ring is a left seminear-ring.
    \item[2)] Let $(S,+)$ be a semigroup and $M(S)$ be the set of maps from $S$ to itself. Moreover, let $+$ and $\cdotp$ the binary operations on $M(S)$ given by $(f+g)(x):=f(x)+g(x)$ and $(f\cdotp g)(x):=g(f(x))$ for all $x\in S,$ and $f,g\in M(S).$ Then, the triple $(M(S),+,\cdotp)$ is a left seminear-ring having the identity function as unit. Moreover, any constant function is a right zero.
\end{itemize}

\end{exs}

\begin{defin}
        Let $(S,+,\cdotp)$ be a left seminear-ring. A subset $T$ of $S$ is said to be a \emph{left subseminear-ring} if it is closed under the operations $+$ and $\cdotp$.
\end{defin}

\noindent The left subseminear-ring provided in the following proposition is of crucial importance for the goal of this section.

\begin{prop}
    Let $S$ be a left seminear ring and $o$ a right zero of $(S,+)$. Then, the set $S_o:=\{o\cdotp a | a\in S\}$ is a left subseminear-ring of $S$ containing $o$
\end{prop}

\begin{proof}
    Straightforward.
\end{proof}

\noindent To give the desired correspondence, we introduce the concept of exponential of left seminear-ring. As one can expect, our definition is consistent with the one given by Rump in \cite[Definition $6$]{rump2019set} for left near-rings. 

\begin{defin}
    Let $S$ be a left seminear-ring. An \emph{exponential} of $S$ is a pair $(o,e)$ where $o$ is an additively idempotent right zero and $e$ a homomorphism from $(S_o,+)$ to $(S,+)$ with $o\cdotp e(s)=s$ for all $s\in S$, $1+e(o)=1$ and such that the elements $1+e(s)$, with $s\in S$, form a subgroup of $(S,\cdotp)$.
\end{defin}

\noindent Now we are able to show the main results of the section. At first, we show that left seminear-rings provide left semi-braces.


\begin{prop}
    Let $(S,+,\cdotp)$ be a unital left seminear-ring and $(o,e)$ be an exponential of $S$. Moreover, suppose that $(S,+)$ is a left cancellative semigroup. Then, the triple $(S_o,+,\circ)$, where the operation $\circ$ is given by $s\circ t:=s\cdotp (1+e(t))$ for all $s,t\in S_o$, is a left cancellative left semi-brace.
\end{prop}

\begin{proof}
    Since $(o,e)$ is exponential, if $s,t\in S_o$, there exists an element that we indicate by $s\circ' r$ such that $1+e(s\circ' r)=(1+e(s))\cdotp (1+e(r))=1+e(s)+(1+e(s))\cdotp e(r)$. Moreover, since $(S,+)$ is left cancellative we have that $e(s\circ' r)= e(s)+(1+e(s))\cdotp e(r)$, therefore it follows that
    \begin{eqnarray}
        s\circ'r&=&o\cdotp e(s\circ'r) \nonumber \\
        &=& o\cdotp (e(s)+(1+e(s))\cdotp e(r)) \nonumber \\
        &=& o\cdotp e(s)+(o\cdotp 1+o\cdotp e(s))\cdotp e(r)) \nonumber \\
        &=& s+s\cdotp e(r) \nonumber \\
        &=& s\cdotp (1+e(r)) \nonumber 
    \end{eqnarray}
where the fourth equality follows because $o\cdotp 1$, which is equal to the idempotent $o$, is a left identity of the left cancellative semigroup $(S,+)$. Now, if we set $s\circ r:=s\circ' r$ for all $r,s\in S_o$ we obtain that
    \begin{eqnarray}
        (s\circ r)\circ t=(s\cdotp (1+e(r)))\cdotp (1+e(t)) &=& (o\cdotp (1+e(s))\cdotp (1+e(r)))\cdotp (1+e(t)) \nonumber \\
        &=& o\cdotp (1+e(s))\cdotp (((1+e(r))\cdotp (1+e(t))) \nonumber \\
        &=& o\cdotp (1+e(s))\cdotp (1+e(r\circ t))\nonumber \\
        &=& s\circ (r\circ t) \nonumber
    \end{eqnarray}
    for all $r,s,t\in S_o$, hence $(S_o,\circ)$ is a semigroup. Moreover, $o\circ r=o\cdotp (1+e(r))=r$ and $r\circ o=r\cdotp (1+e(o))=r$ for all $r\in S_o$, therefore $(S_o,\circ)$ is a monoid having $o$ as neutral element. Since $(o,e)$ is an exponential, for every $s\in S_o$ there exist an element $s^-\in S_o$ such that $(1+e(s))\cdotp (1+e(s^-))=(1+e(s^-))\cdotp (1+e(s))=1$, and using this fact one can easily show that $(S_o,\circ)$ is a group. Since $(S,+)$ is left cancellative, then so $(S_o,+)$ is. Finally, if $r,s,t\in S_o$, we have that 
     \begin{eqnarray}
       s\circ (r+t)=s\cdotp (1+e(r+t)) &=& s\cdotp (1+e(r))+s\cdotp e(t) \nonumber \\
        &=& s\cdotp (1+e(r))+o+s\cdotp e(t) \nonumber \\
        &=& s\cdotp (1+e(r))+s\cdotp (1+e(s^-))+s\cdotp e(t)\nonumber \\
        &=& s\circ r+s\circ (s^-+t) \nonumber 
    \end{eqnarray}
    hence the thesis follows.
    \end{proof}

\noindent As a kind of converse of the previous proposition, we show that every left cancellative left semi-brace comes from a suitable left seminear-ring.

    \begin{prop}
        Let $(A,+,\circ)$ be a left cancellative left semi-brace. Then, there exist a left seminear-ring $(S,+,\cdotp)$ with an exponential $(o,e)$ such that $(A,+,\circ)\cong (S_o,+,\circ)$.
    \end{prop}

    \begin{proof}
        Let $(S,+,\cdotp)$ be the left seminear-ring on $M(A)$ defined as in 2) of \cref{essemi} using $(A,+)$.
        Clearly, $(A,+)$ can be identified with the set of the constant maps of $S$ and in this way $A$ coincides with $S_0$. Now, if we set $e:A\longrightarrow S$, $a\mapsto e(a)$, where $e(a)$ is the map given by $e(a)(b):=\lambda_b(a)$ for all $a,b\in A$, we obtain that $e$ is a homomorphism from $(A,+)$ to $(S,+)$ and
        $$a=a\circ 0 =a+\lambda_a(0)=(1+e(0))(a)$$
        for all $a\in A$, therefore $1+e(0)=1$. Now, by associativity of $(A,\circ)$ we obtain $(1+e(a\circ b))=(1+e(a))\cdotp (1+e(b))$, hence the elements of the form $(1+e(a))$ form a subgroup of $(S,\cdotp)$. Moreover, $0\cdotp e(a)(b)=e(a)(0)=\lambda_0(a)=a$ for all $a\in A$, hence $(e,0)$ is an exponential of $S$. Finally, since $a\circ b$ is equal to $a+\lambda_a(b)$ and by the identification $A=S_0$ we obtain that $a\circ b=a\cdotp (1+e(b))$ for all $a,b\in A$, therefore $(A,+,\circ)$ is isomorphic to the left cancellative semi-brace $(S_0,+,\circ)$.
    \end{proof}

\noindent Therefore, the previous propositions showed the following correspondence theorem.

    \begin{theor}
         Let $(S,+,\circ)$ be a left seminear-ring and $(o,e)$ be an exponential. Then, the triple $(S_o,+,\circ)$, where the operation $\circ$ is given by $s\circ t:=s\cdotp (1+e(t))$ for all $s,t\in S_o$, is a left cancellative left semi-brace. Conversely, every left cancellative left semi-brace arises in this way.
    \end{theor}

As a final remark of the section, we highlight that the previous theorem, replacing left seminear-ring by left near-ring and left cancellative left semi-brace by skew left brace, allow to recover Rump's correspondence given in \cite[Theorem 1]{rump2019set}.

\section{Left semi-braces with $G=Ker(\lambda_{|_E})$ and groups semidirect products}

In this section, we study the structure of left cancellative left semi-braces where the group $G$ acts trivially on the additive idempotents by the map $\lambda$. As a main result, we show such a left semi-braces are completely determined by a skew left brace and a groups semidirect product.

\medskip

\noindent By \cref{lambdasemi}, $\lambda$ induces an action of $(B,\circ)$ on $E$. From now on, we indicate this action by $\lambda_{|_E}$. We start showing that the kernel of $\lambda_{|_E}$ is contained in $G$.

\begin{lemma}\label{lem1}
Let $B$ be a left cancellative left semi-braces. Then, the kernel of $\lambda_{|_E}$ is contained in $G$.
\end{lemma}

\begin{proof}
Let $b$ an element of $Ker(\lambda_{|_E})$. By \cref{fatt}, there exist $g\in G, e\in E$ such that $b=g\circ e$. Then $\lambda_{g\circ e}(f)=\lambda_g\lambda_e(f)=f$ for all $f\in E$, and this implies that $0=\lambda_g^{-1}(0)=\lambda_e(0)=e$, hence $Ker(\lambda_{|_E}) $ is contained in $G$.
\end{proof}

\noindent Now, we want to provide a structure theorem for the left cancellative left semi-braces having $G$ equal to $Ker(\lambda_{|_E})$. At first, we need a preliminary lemma.

\begin{lemma}\label{lemker1}
Let $B$ be a left cancellative left semi-braces, $e$ an element of $E$ and suppose that $G=Ker(\lambda_{|_E})$. Then, the map $\alpha_e:G\longrightarrow G$, $g\mapsto e\circ g\circ e^{-}$ is an element of $Aut(G,+,\circ)$.
\end{lemma}

\begin{proof}
Since $G$ is equal to the kernel of the groups homomorphism $\lambda_{|_E}:(B,\circ)\longrightarrow Aut(E,+) $, it is a normal subgroup of $(B,\circ)$ and  clearly $\alpha_e\in Aut(G,\circ)$. Therefore it is sufficient to show that $\alpha_e(g_1+g_2)=\alpha_e(g_1)+\alpha_e(g_2)$ for all $g_1,g_2\in G$. Indeed
\begin{eqnarray}
e\circ (g_1+g_2)\circ e^{-} &=& e\circ (g_1+g_2+\lambda_{g_1+g_2}(e^{-})) \nonumber \\
&=& e\circ (g_1+g_2+e^{-}) \nonumber \\
&=& e\circ (g_1+e^-+g_2+e^{-}) \nonumber \\
&=& e\circ (g_1\circ \lambda_{g_1^-}(e^-)+g_2\circ \lambda_{g_2^-}(e^{-})) \nonumber \\
&=& e\circ (g_1\circ e^-+g_2\circ e^{-}) \nonumber \\
&=&  e\circ g_1\circ e^- + e\circ (e^-+g_2\circ e^{-}) \nonumber \\
&=&  e\circ g_1\circ e^- + e\circ g_2\circ e^{-} \nonumber 
\end{eqnarray}
for all $g_1,g_2\in G$, hence the thesis.
\end{proof}

\noindent The previous lemma states that, given a left cancellative semi-brace $B$ such that $G=Ker(\lambda_{|_E})$, one can define a standard semidirect product of left semi-brace of $G$ by $E$ using the map $\alpha:E\longrightarrow Aut(G,+,\circ)$, $e\mapsto \alpha_e$. The next theorem shows that this semidirect product completely determines the structure of $B$.

\begin{theor}\label{teostrutt}
Let $B$ be a left cancellative left semi-braces. Then, $G$ is equal to $Ker(\lambda_{|_E})$ if and only if $B$ is isomorphic to a semidirect product of left semi-braces $A\rtimes T$ of a skew left brace $A$ by a trivial left semi-brace $T$.
\end{theor}

\begin{proof}
Let $B$ be a left cancellative left semi-brace such that $G=Ker(\lambda_{|_E})$. Let $f$ be the map from $B$ to $G\rtimes_{\alpha} E$ given by $f(g+e):=(g,e)$ for all $g\in G$, $e\in E$. Clearly, $f$ is bijective and $$f((g_1+e_1)+(g_2+e_2))=f(g_1+g_2+e_2)=(g_1+g_2,e_2)$$
$$f(g_1+e_1)+f(g_2+e_2)=(g_1,e_1)+(g_2,e_2)=(g_1+g_2,e_2) $$
for all $g_1,g_2\in G$, $e_1,e_2\in E$, hence $f$ is an isomorphism between the additive semigroups. Moreover, we have that 
\begin{eqnarray}
(g_1+e_1)\circ (g_2+e_2)&=& g_1\circ \lambda_{g_1^{-}}(e_1)\circ g_2\circ  \lambda_{g_2^{-}}(e_2) \nonumber \\
&=& g_1\circ e_1\circ g_2\circ e_2 \nonumber \\
&=& g_1\circ e_1\circ g_2\circ e_1^{-}\circ e_1 \circ e_2 \nonumber \\
&=& (g_1\circ e_1\circ g_2\circ e_1^{-}) + \lambda_{(g_1\circ e_1\circ g_2\circ e_1^{-})^-}(e_1 \circ e_2) \nonumber \\
&=& (g_1\circ \alpha_{e_1}(g_2)) + (e_1 \circ e_2) \nonumber 
\end{eqnarray}
for all $g_1,g_2\in G$, $e_1,e_2\in E$, and this clearly implies that $f$ is an isomorphism from $(B,\circ)$ to the multiplicative group of $G\rtimes_{\alpha} E$, therefore the necessary condition follows.\\
Conversely, if $B\cong  A\rtimes T$ for some skew left brace $A$ and a trivial left semi-brace $T$, by a standard calculation one can show that $\{0\}\times T$ is the set of idempotents $E$ and $G$ is equal to $A\times \{0\}$. Moreover, we have that $\lambda_{(g_1,e_1)}(0,e_2)=(0,e_1\circ e_2)$ for all $g_1\in A,$ $e_1,e_2\in E$, therefore it follows that $Ker(\lambda_{|_E})=A\times \{0\}$.
\end{proof}

\noindent Now, we conclude the section by a further characterization of left cancellative left semi-braces having $G=Ker(\lambda_{|_E})$. Indeed, we show that all these left semi-braces come from a particular semidirect product of groups.

\begin{cor}
Let $B$ be a left cancellative left semi-braces such that $G=Ker(\lambda_{|_E})$. Then, $G$ is a skew left brace such that $(G,\circ)$ is a normal subgroup of $(B,\circ)$, the intersection $G\cap E$ is equal to $\{0\}$, the automorphism of $(G,\circ)$ given by $g\mapsto e\circ g\circ e^-$ is also an automorphism of $(G,+)$ for all $e\in E$ and $(B,\circ)$ is equal to the semidirect product of $(G,\circ)$ and $(E,\circ)$.\\
Conversely, every groups semidirect product $A=H\rtimes_{\beta} F$ together with a skew left brace structure on $H$ such that $\beta_f\in Aut(H,+,\circ)$ for all $f\in F$ give rise to a left cancellative left semi-brace on $A$ such that $H\times \{0\}$ is equal to $Ker(\lambda_{|_F})$ and $\{0\}\times F$ is the set of idempotents.
\end{cor}

\begin{proof}
The first part follows by \cref{teostrutt}. Conversely, suppose that $A$ is equal to $H\rtimes_{\beta} F$, where $H$, $F$ and $\beta$ are as in the hypothesis. Let $\oplus$ the binary operation on $H\times F$ given by
$$(h_1,f_1)\oplus (h_2,f_2):=(h_1+h_2,f_2) $$
for all $h_1,h_2\in H$, $f_1,f_2\in F$. Since $(A,\oplus,\circ)$ coincides with the semidirect product of the skew left brace $(H,+,\circ)$ and the trivial left semi-brace on $(F,\circ)$, it is a left cancellative left semi-brace. By an easy calculation, one can show that $H\times \{0\}$ is equal to $Ker(\lambda_{|_F})$ and $\{0\}\times F$ is the set of idempotents.
\end{proof}

\section{Left cancellative left semi-braces of size $pq$}

In this section we classify left cancellative left semi-braces of size $pq$, where $p,q$ are prime numbers such that $q\leq p$, and with a non-trivial set of additive idempotents. The following lemma is simple but of crucial importance for our scope.

\begin{lemma}\label{prelipq}
Let $B$ a left cancellative left semi-brace such that $|B|=pq$ and suppose that $|E|>1$. Then, $E$ is an ideal of $B$ or $G=Ker(\lambda_{|_E})$.
\end{lemma}

\begin{proof}
If $p=q$, then by $(B,\circ) $ is abelian and this implies that $E$ is a normal subgroup of $(B,\circ)$, hence an ideal. If $p>q$ and $E=p$ then it is again an ideal of $B$. Finally, if $p>q$ and $|E|=q$, by \cref{fatt} we have that $|G|=p$, therefore $|\lambda_{|_E}(G)|\in \{1,p\}$. Since $|\lambda_{|_E}(G)|$ must divides $q!$, necessarily $|\lambda_{|_E}(G)|=1$ and this fact, together with \cref{lem1}, implies $G=Ker(\lambda_{|_E})$
\end{proof}

\noindent Now, we have to consider some cases.

\subsection{Case $p\not\equiv_q 1$}

\smallskip

As we will show below, in this the classification is quite simple, provided a preliminary lemma.

\begin{lemma}\label{classlem}
Let $B$ be a left cancellative left semi-brace such that $|B|=pq$, where $p,q$ are prime numbers (not necessarily distinct) such that $p\not\equiv_q 1$. Moreover, suppose that $|E|>1$. Then $B$ is isomorphic to the direct product of the trivial left semi-braces $E$ and the trivial skew left brace $G$.
\end{lemma}

\begin{proof}
Since $p,q$ are prime numbers such that $p\not\equiv_q 1$, we have that $(B,\circ)$ is abelian. Hence by \cref{consesemi} $B$ is isomorphic to the direct product of the trivial left semi-brace on $(E,\circ)$ and the skew left brace on $(G,\circ)$, which is a trivial skew brace since $|G|\in \{1,p,q\}$. By a standard calculation on can show that the isomorphism class of this left semi-brace is completely determined by the size of $E$.
\end{proof}

\begin{theor}
Let $B$ be a left cancellative left semi-brace such that $|B|=pq$, where $p,q$ are prime numbers (not necessarily distinct) such that $p\not\equiv_q 1$. Moreover, suppose that $|E|>1$. Then, up to isomorphism, $B$ is one of these following left semi-braces:
\begin{itemize}
    \item[1)] $(B,\circ)=\mathbb{Z}/p^2\mathbb{Z}$, $E=B$ and $a+b:=b $ for all $a,b\in B$;
    \item[2)] $(B,\circ)=\mathbb{Z}/p\mathbb{Z}\times \mathbb{Z}/p\mathbb{Z}$, $E=B$ and $a+b:=b $ for all $a,b\in B$;
    \item[3)] $(B,\circ)=\mathbb{Z}/p\mathbb{Z}\times \mathbb{Z}/q\mathbb{Z}$, $E=B$ and $a+b:=b $ for all $a,b\in B$ (where $p\neq q$);
    \item[4)] $(B,\circ)=\mathbb{Z}/p\mathbb{Z}\times \mathbb{Z}/q\mathbb{Z}$, $E=\{0\}\times \mathbb{Z}/q\mathbb{Z}$ and $(a,b)+(c,d):=(a+c,d) $ for all $(a,b),(c,d)\in B$ (where $p\neq q$);
    \item[5)] $(B,\circ)=\mathbb{Z}/q\mathbb{Z}\times \mathbb{Z}/p\mathbb{Z}$, $E=\{0\}\times \mathbb{Z}/p\mathbb{Z}$ and $(a,b)+(c,d):=(a+c,d) $ for all $(a,b),(c,d)\in B$ (where $p\neq q$);
    \item[6)] $(B,\circ)=\mathbb{Z}/p\mathbb{Z}\times \mathbb{Z}/p\mathbb{Z}$,  $E=\{0\}\times \mathbb{Z}/p\mathbb{Z}$ and $(a,b)+(c,d):=(a+c,d) $ for all $(a,b),(c,d)\in B$;
\end{itemize}
\end{theor}

\begin{proof}
The thesis follows applying \cref{classlem} and a standard calculation.
\end{proof}


\subsection{Case $p\equiv_q 1$}

\smallskip

In this subsection we focus on the case $p\equiv_q 1$. At first, we need a series of preliminary lemmas.

\begin{lemma}\label{semipq1}
Let $B$ be a left cancellative left semi-brace such that $|B|=pq$, where $p,q$ are prime numbers such that $p\equiv_q 1$. Moreover, suppose that $|E|=q$. Then $G=Ker(\lambda_{|_E}) $ and $B$ is isomorphic to the semidirect product of the trivial left semi-braces $E$ and the trivial skew left brace $G$ by the inner action of $(E,\circ)$ on $(G,\circ)$.
\end{lemma}

\begin{proof}
It follows by \cref{teostrutt} and \cref{prelipq}.
\end{proof}

\begin{lemma}\label{semipq2}
Let $B$ be a left cancellative left semi-brace such that $|B|=pq$, where $p,q$ are prime numbers such that $p\equiv_q 1$. Moreover, suppose that $|E|=p$. Then $B$ is isomorphic to a semidirect product of the trivial skew left brace $G$ and the trivial left semi-brace $G$.
\end{lemma}

\begin{proof}
The thesis follows by \cref{semidi} and \cref{prelipq}.
\end{proof}

Before giving the following lemma, recall that if $(B_1,\circ),(B_2,\circ')$ are groups and $\alpha:(B_2,\circ')\longrightarrow Aut(B_1,\circ)$ a homomorphism from $B_2$ to the automorphisms groups of $(B_1,\circ)$, then we provide both a semidirect product of groups $B_1\rtimes B_2$ and a semidirect product of left semi-braces $B_1\rtimes B_2$ where $B_1$ can be a trivial left semi-brace or a trivial skew left brace.

\begin{lemma}\label{semipq3}
Let $p,q$ be prime numbers such that $p\equiv_q 1$, $E$ a trivial left semi-brace and $G$ a trivial skew left brace. Then:
\begin{itemize}
    \item[1)] If $|E|=q$, $|G|=p$ and $\alpha,\alpha'$ are homomorphisms from $(E,\circ)$ to $Aut(G,\circ)$ such that the semidirect products of groups $G\rtimes_{\alpha} E$ and $G\rtimes_{\alpha'} E$ are isomorphic, then the semidirect products of left semi-braces $G\rtimes_{\alpha} E $ and $G\rtimes_{\alpha'} E$ are isomorphic.
    \item[2)] If $|E|=p$, $|G|=q$ and $\alpha,\alpha'$ are homomorphisms from $(G,\circ)$ to $Aut(E,\circ)$ such that the semidirect products of groups $E\rtimes_{\alpha} G$ and $E\rtimes_{\alpha'} G$ are isomorphic, then the semidirect products of left semi-braces $E\rtimes_{\alpha} G $ and $E\rtimes_{\alpha'} G$ are isomorphic.
\end{itemize}
\end{lemma}

\begin{proof}
For 1), by a standard calculation one can show that an isomorphism between the semidirect products of groups $G\rtimes_{\alpha} E$ and $G\rtimes_{\alpha'} E$ is also an isomorphism of the semidirect products of left semi-braces $G\rtimes_{\alpha} E $ and $G\rtimes_{\alpha'} E$. The second assumption can be shown in the same way.
\end{proof}

\noindent Now, we can give the theorem which closes the section.

\begin{theor}
Let $B$ be a left cancellative left semi-brace such that $|B|=pq$, where $p,q$ are prime numbers such that $p\equiv_q 1$. Moreover, suppose that $|E|>1$ and let $\alpha$ be an injective homomorphism from $\mathbb{Z}/q\mathbb{Z}$ to $Aut(\mathbb{Z}/p\mathbb{Z},+)$. Then, up to isomorphism, $B$ is one of these following left semi-braces:
\begin{itemize}
    \item[1)] $(B,\circ)=\mathbb{Z}/p\mathbb{Z}\rtimes_{\alpha} \mathbb{Z}/q\mathbb{Z}$, $E=B$ and $a+b:=b $ for all $a,b\in B$;
    \item[2)] $(B,\circ)=\mathbb{Z}/p\mathbb{Z}\times \mathbb{Z}/q\mathbb{Z}$, $E=B$ and $a+b:=b $ for all $a,b\in B$;
    \item[3)] $(B,\circ)=G\rtimes_{\alpha} E$, $(E,\circ)=\mathbb{Z}/q\mathbb{Z}$, $(G,\circ)=\mathbb{Z}/p\mathbb{Z}$, $(g_1,e_1)+(g_2,e_2)=(g_1+g_2,e_2)$ and $(g_1,e_1)\circ (g_2,e_2)=(g_1+ \alpha_{e_1}(g_2),e_1+ e_2)$ for all $(g_1,e_1),(g_2,e_2)\in B$;
    \item[4)] $(B,\circ)=E\rtimes_{\alpha} G$, $(E,\circ)=\mathbb{Z}/p\mathbb{Z}$, $(G,\circ)=\mathbb{Z}/q\mathbb{Z}$, $(e_1,g_1)+(e_2,g_2)=(e_2,g_1+g_2)$ and $(e_1,g_1)\circ (e_2,g_2)=(e_1+ \alpha_{g_1}(e_2),g_1+ g_2)$ for all $(g_1,e_1),(g_2,e_2)\in B$;
    \item[5)] $(B,\circ)=G\times E$, $(E,\circ)=\mathbb{Z}/q\mathbb{Z}$, $(G,\circ)=\mathbb{Z}/p\mathbb{Z}$, $(g_1,e_1)+(g_2,e_2)=(g_1+g_2,e_2)$ and $(g_1,e_1)\circ (g_2,e_2)=(g_1+ g_2,e_1+ e_2)$ for all $(g_1,e_1),(g_2,e_2)\in B$;
    \item[6)] $(B,\circ)=G\times E$, $(E,\circ)=\mathbb{Z}/p\mathbb{Z}$, $(G,\circ)=\mathbb{Z}/q\mathbb{Z}$, $(g_1,e_1)+(g_2,e_2)=(g_1+g_2,e_2)$ and $(g_1,e_1)\circ (g_2,e_2)=(g_1+ g_2,e_1+ e_2)$ for all $(g_1,e_1),(g_2,e_2)\in B$.
\end{itemize}
\end{theor}

\begin{proof}
If $|E|=pq$, then $B$ is a trivial left semi-brace, hence we find the left semi-braces 1), if $(B,\circ)$ is not abelian, and 2), if $(B,\circ)$ is abelian. If $1<|E|<pq$, by \cref{semipq1} and \cref{semipq2} $B$ is a semidirect product of two left semi-braces. If $|E|=p$, by \cref{semipq3} we find the left semi-brace 4), if $(B,\circ)$ is not abelian, and the left semi-brace 6), if $(B,\circ)$ is abelian. If $|E|=q$, again by \cref{semipq3}, we find the left semi-brace 3), if $(B,\circ)$ is not abelian, and the left semi-brace 5), if $(B,\circ)$ is abelian. 
\end{proof}

\section{Left cancellative left semi-brace of size $2p^2$}

In this section, we classify left cancellative left semi-braces of size $2p^2$, where $p$ is a prime number with $p>2$, and such that $E$ is a Sylow subgroup of $(B,\circ)$. As for the left semi-braces of size $pq$, we have to consider some cases.



\subsection{Case $|E|=2$ and $(G,\circ)$ cyclic}

\begin{lemma}\label{e11}
Let $B$ be a left cancellative left semi-brace of size $2p^2$ and suppose that $|E|=2$. Then, $B$ is isomorphic to the semidirect product of the trivial left semi-braces $E$ and a skew left brace $G$ of size $p^2$. 
\end{lemma}

\begin{proof}
If $|E|=2$, by \cref{fatt}, we have that $|G|=p^2$, therefore $G$ is a normal subgroup of $(B,\circ)$. Moreover, since $\lambda_g(0)=0$ for all $g\in G$, necessarily $\lambda_g(e)=e$ for all $g\in G$, $e\in E$, hence $G=Ker(\lambda_{|_E})$. Therefore the thesis follows by \cref{teostrutt}.
\end{proof}

\begin{lemma}\label{e12}
There are two skew left braces of size $p^2$ and with cyclic multiplicative group. The skew left brace $G_1:=(\mathbb{Z}/p^2\mathbb{Z},+,\circ)$ given by $g_1\circ g_2:=g_1+g_2$ for all $g_1,g_2\in G_1$ and the skew left brace $G_2:=(\mathbb{Z}/p^2\mathbb{Z},+,\circ)$ given by $g_1\circ g_2:=g_1+g_2+pg_1 g_2$ for all $g_1,g_2\in G_1$. Moreover, $(G_1,+,\circ)$ has exactly one automorphism having order $2$, while $(G_2,+,\circ)$ has not such an automorphism.
\end{lemma}

\begin{proof}
It follows by \cite[Proposition 2.4]{bachiller2015classification} and \cite[Lemma 3]{dietzel2021braces}.
\end{proof}

\begin{theor}
Let $B$ be a left cancellative left semi-brace of size $2p^2$. Suppose that $|E|=2$, $(G,\circ)$ is cyclic and let $\alpha$ be the injective homomorphism from $(\mathbb{Z}/2\mathbb{Z},+)$ to $Aut(\mathbb{Z}/p^2\mathbb{Z},+)$. Then, up to isomorphisms, $B$ is one of the following left semi-braces:
\begin{itemize}
     \item[1)] $(B,\circ)=\mathbb{Z}/p^2\mathbb{Z}\rtimes_{\alpha} \mathbb{Z}/2\mathbb{Z}$, $E=\{0\}\times \mathbb{Z}/2\mathbb{Z}$, $G=G_1=\mathbb{Z}/p^2\mathbb{Z}\times \{0\}$,  $(g_1,e_1)+(g_2,e_2)=(g_1+g_2,e_2)$ and $(g_1,e_1)\circ (g_2,e_2)=(g_1+ \alpha_{e_1}(g_2),e_1+ e_2)$ for all $(g_1,e_1),(g_2,e_2)\in B$;
     \item[2)] $(B,\circ)=\mathbb{Z}/p^2\mathbb{Z}\times \mathbb{Z}/2\mathbb{Z}$, $E=\{0\}\times \mathbb{Z}/2\mathbb{Z}$, $G=G_1=\mathbb{Z}/p^2\mathbb{Z}\times \{0\}$,  $(g_1,e_1)+(g_2,e_2)=(g_1+g_2,e_2)$ and $(g_1,e_1)\circ (g_2,e_2)=(g_1+ g_2,e_1+ e_2)$ for all $(g_1,e_1),(g_2,e_2)\in B$;
     \item[3)] $(B,\circ)=\mathbb{Z}/p^2\mathbb{Z}\times \mathbb{Z}/2\mathbb{Z}$, $E=\{0\}\times \mathbb{Z}/2\mathbb{Z}$, $G=G_2=\mathbb{Z}/p^2\mathbb{Z}\times \{0\}$,  $(g_1,e_1)+(g_2,e_2)=(g_1+g_2,e_2)$ and $(g_1,e_1)\circ (g_2,e_2)=(g_1+ g_2+p g_1 g_2,e_1+ e_2)$ for all $(g_1,e_1),(g_2,e_2)\in B$.
\end{itemize}
\end{theor}

\begin{proof}
By \cref{e11} $B$ is isomorphic to the semidirect product of the trivial left semi-brace $E$ and the skew left brace $G$. If $G=G_2$, by \cref{e12} $B$ is necessarily isomorphic to the left semi-brace 3). If $G=G_1$ then $B$ is isomorphic to the left semi-brace 1), if $(B,\circ)$ is not abelian, and to the left semi-brace 2), if $(B,\circ)$ is abelian. Clearly, all these left semi-braces are not isomorphic each other.
\end{proof}

\subsection{Case $|E|=2$ and $(G,\circ)$ abelian non-cyclic}

\begin{lemma}\label{e14}
There are two skew left braces of size $p^2$ and with non-cyclic multiplicative group. The skew left brace $G_3:=(\mathbb{Z}/p\mathbb{Z}\times \mathbb{Z}/p\mathbb{Z},+,\circ)$ given by $g_1\circ g_2:=g_1+g_2$ for all $g_1,g_2\in G_3$ and the skew left brace $G_4:=(\mathbb{Z}/p\mathbb{Z}\times \mathbb{Z}/p\mathbb{Z},+,\circ)$ given by $(g_1,f_1)\circ (g_2,f_2):=(g_1+g_2+f_1 f_2,f_1+f_2)$ for all $g_1,g_2,f_1,f_2\in \mathbb{Z}/p\mathbb{Z}$. 
\end{lemma}

\begin{proof}
It follows by \cite[Proposition 2.4]{bachiller2015classification}.
\end{proof}

\begin{lemma}\label{e15}
Up to conjugacy, the automorphisms $\alpha_1$ and $\alpha_2$ of order $2$ of the skew left brace $(G_3,+,\circ)$ can be identified, respectively, with the elements of $GL_2(p)$ given by
$$
A_1:=\begin{bmatrix}
-1 & 0  \\
0 & -1 
\end{bmatrix}
\mbox{ and }
A_2:=\begin{bmatrix}
1 & 0  \\
0 & -1
\end{bmatrix}
$$
having order $2$ in $GL_2(p)$.
The automorphisms of order $2$ of the skew left brace $(G_4,+,\circ)$ can be identified with the elements of $GL_2(p)$ given by
$$
A_b:=\begin{bmatrix}
1 & b  \\
0 & -1 
\end{bmatrix}
$$
where $b\in \mathbb{Z}/p\mathbb{Z}$. Moreover, for any $c\in \mathbb{Z}/p\mathbb{Z}$ there exist $d\in \mathbb{Z}/p\mathbb{Z}$ such that $A_c=A_d^{-1}A_0 A_d$. 
\end{lemma}

\begin{proof}
Since $G_3$ is a trivial skew left brace, its automorphisms group is just the automorphisms group of $(\mathbb{Z}/p\mathbb{Z}\times \mathbb{Z}/p\mathbb{Z},+)$, and hence the first part follows. The last part follows by \cite[Lemma 3]{dietzel2021braces} and a standard calculation.
\end{proof}

\noindent Before giving the following lemma, we have to fix a notation. If $(B_1,+,\circ),(B_2,+',\circ')$ are left cancellative left semi-braces and $\alpha:(B_2,\circ')\longrightarrow Aut(B_1,+,\circ)$ is a homomorphism from $B_2$ to the automorphisms groups of $(B_1,+,\circ)$, then we indicate by $B_1\rtimes_{\alpha} B_2$ the semidirect product of the left semi-braces $B_1$ and $B_2$ by $\alpha$ and by $(B_1\rtimes_{\alpha} B_2,\circ)$ the semidirect product of the groups $(B_1,\circ)$ and $(B_2,\circ')$ by $\alpha$.

\begin{lemma}\label{semipq4}
Let $p$ be a prime number greater than $2$, $E$ the trivial left semi-brace having size $2$ and $G$ a skew left brace having size $p^2$ such that $(G,\circ)$ is the abelian non-cyclic group. Then if $\alpha,\alpha'$ are homomorphisms from $(E,\circ)$ to $Aut(G,+,\circ)$ such that the semidirect products of groups $(G\rtimes_{\alpha} E,\circ)$ and $(G\rtimes_{\alpha'} E,\circ)$ are isomorphic, then the semidirect products of left semi-braces $G\rtimes_{\alpha} E $ and $G\rtimes_{\alpha'} E$ are isomorphic.
\end{lemma}

\begin{proof}
In the same spirit of \cref{semipq3}, using \cref{e15}, one can easily show that there is a group isomorphism from $(G\rtimes_{\alpha} E,\circ)$ to $(G\rtimes_{\alpha'} E,\circ)$ that give rise to a left semi-braces isomorphisms.
\end{proof}

\begin{theor}
Let $B$ be a left cancellative left semi-brace of size $2p^2$. Suppose that $|E|=2$ and $(G,\circ)$ is the abelian non-cyclic group having order $p^2$. Then, up to isomorphisms, $B$ is one of the following left semi-braces:
\begin{itemize}
     \item[1)] $(B,\circ)=(\mathbb{Z}/p\mathbb{Z}\times \mathbb{Z}/p\mathbb{Z})\times \mathbb{Z}/2\mathbb{Z}$, $E=\{0\}\times \mathbb{Z}/2\mathbb{Z}$, $G=G_3=(\mathbb{Z}/p\mathbb{Z}\times \mathbb{Z}/p\mathbb{Z})\times \{0\}$,  $(g_1,e_1)+(g_2,e_2)=(g_1+g_2,e_2)$ and $(g_1,e_1)\circ (g_2,e_2)=(g_1+ g_2,e_1+ e_2)$ for all $(g_1,e_1),(g_2,e_2)\in B$;
     \item[2)] $(B,\circ)=(\mathbb{Z}/p\mathbb{Z}\times \mathbb{Z}/p\mathbb{Z})\rtimes_{\alpha_1} \mathbb{Z}/2\mathbb{Z}$, $E=\{0\}\times \mathbb{Z}/2\mathbb{Z}$, $G=G_3=(\mathbb{Z}/p\mathbb{Z}\times \mathbb{Z}/p\mathbb{Z})\times \{0\}$,  $(g_1,e_1)+(g_2,e_2)=(g_1+g_2,e_2)$ and $(g_1,e_1)\circ (g_2,e_2)=(g_1+ \alpha_{1_{e_1}}(g_2),e_1+ e_2)$ for all $(g_1,e_1),(g_2,e_2)\in B$;
     \item[3)] $(B,\circ)=(\mathbb{Z}/p\mathbb{Z}\times \mathbb{Z}/p\mathbb{Z})\rtimes_{\alpha_2} \mathbb{Z}/2\mathbb{Z}$, $E=\{0\}\times \mathbb{Z}/2\mathbb{Z}$, $G=G_3=(\mathbb{Z}/p\mathbb{Z}\times \mathbb{Z}/p\mathbb{Z})\times \{0\}$,  $(g_1,e_1)+(g_2,e_2)=(g_1+g_2,e_2)$ and $(g_1,e_1)\circ (g_2,e_2)=(g_1+ \alpha_{2_{e_1}}(g_2),e_1+ e_2)$ for all $(g_1,e_1),(g_2,e_2)\in B$;
     \item[4)] $(B,\circ)=(\mathbb{Z}/p\mathbb{Z}\times \mathbb{Z}/p\mathbb{Z})\times \mathbb{Z}/2\mathbb{Z}$, $E=\{0\}\times \mathbb{Z}/2\mathbb{Z}$, $G=G_4=(\mathbb{Z}/p\mathbb{Z}\times \mathbb{Z}/p\mathbb{Z})\times \{0\}$,  $(g_1,h_1,e_1)+(g_2,h_2,e_2)=(g_1+g_2,h_1+h_2,e_2)$ and $(g_1,e_1)\circ (g_2,e_2)=(g_1+ g_2+h_1h_2,h_1+h_2,e_1+ e_2)$ for all $(g_1,e_1),(g_2,e_2)\in B$;
     \item[5)] $(B,\circ)=(\mathbb{Z}/p\mathbb{Z}\times \mathbb{Z}/p\mathbb{Z})\rtimes_{\beta} \mathbb{Z}/2\mathbb{Z}$, $E=\{0\}\times \mathbb{Z}/2\mathbb{Z}$, $G=G_4=(\mathbb{Z}/p\mathbb{Z}\times \mathbb{Z}/p\mathbb{Z})\times \{0\}$,  $(g_1,h_1,e_1)+(g_2,h_2,e_2)=(g_1+g_2,h_1+h_2,e_2)$ and $(g_1,h_1,e_1)\circ (g_2,h_2,e_2)=((g_1+g_2+(-1)^{e_1}h_1h_2,h_1+(-1)^{e_1}h_2,e_1+ e_2)$ for all $(g_1,h_1,e_1),(g_2,h_2,e_2)\in B$.
\end{itemize}
where $\alpha_1$ and $\alpha_2$ are as in \cref{e15}.
\end{theor}

\begin{proof}
By \cref{e11}, $B$ is isomorphic to the semidirect product of the left semi-braces $E$ and $G$. If $(B,\circ)$ is abelian we have that $B$ is the direct product of the left semi-brace $E$ and the skew left brace $G$. Then we find the left semi-braces 1) if $G$ is a trivial skew left brace, and 4) if $G$ is a non-trivial skew left brace. Now, suppose that $(B,\circ)$ is a non-abelian group. If $G=G_3$ and $(B,\circ)$ is a non-abelian group with trivial center, we find the left semi-brace 2), while if $G=G_3$ and $(B,\circ)$ is a non-abelian group with non-trivial center, then we find the left semi-brace 3). Finally, if $G=G_4$ and $(B,\circ)$ is a non-abelian group, by \cref{e15} $(B,\circ)$ is the non abelian group with non-trivial center, hence we find the left semi-brace 5). By \cref{semipq4} all the cases are covered.
\end{proof}

\subsection{Case $|E|=p^2$}

\begin{lemma}\label{e1p2}
Let $B$ be a left cancellative left semi-brace of size $2p^2$ and suppose that $|E|=p^2$. Then, $B$ is isomorphic to the semidirect product of the trivial skew left brace $G$ of size $2$ and the trivial left semi-brace $E$. 
\end{lemma}

\begin{proof}
Since $(E,\circ)$ has index $2$ in $(B,\circ)$, it is a normal subgroup. Hence the thesis follows by \cref{semidi} and \cref{fatt}.  
\end{proof}

\begin{theor}
Let $B$ be a left cancellative left semi-brace of size $2p^2$. Suppose that $|E|=p^2$. Then, up to isomorphisms, $B$ is one of the following left semi-braces:
 \item[1)] $(B,\circ)=(\mathbb{Z}/p\mathbb{Z}\times \mathbb{Z}/p\mathbb{Z})\times \mathbb{Z}/2\mathbb{Z}$, $G=\{0\}\times \mathbb{Z}/2\mathbb{Z}$, $E=(\mathbb{Z}/p\mathbb{Z}\times \mathbb{Z}/p\mathbb{Z})\times \{0\}$,  $(e_1,g_1)+(e_2,g_2)=(e_2,g_1+g_2)$ and $(e_1,g_1)\circ (e_2,g_2)=(e_1+e_2,g_1+ g_2)$ for all $(g_1,e_1),(g_2,e_2)\in B$;
 \item[2)] $(B,\circ)=(\mathbb{Z}/p^2\mathbb{Z})\times\mathbb{Z}/2\mathbb{Z}$, $E=\mathbb{Z}/p^2\mathbb{Z}\times \{0\}$, $G=\{0\}\times \mathbb{Z}/2\mathbb{Z}$,  $(e_1,g_1)+(e_2,g_2)=(e_2,g_1+g_2)$ and $(e_1,g_1)\circ (e_2,g_2)=(e_1+e_2,g_1+ g_2)$ for all $(g_1,e_1),(g_2,e_2)\in B$;
     \item[3)] $(B,\circ)=(\mathbb{Z}/p\mathbb{Z}\times \mathbb{Z}/p\mathbb{Z})\rtimes_{\alpha} \mathbb{Z}/2\mathbb{Z}$, $E=\mathbb{Z}/p\mathbb{Z}\times \mathbb{Z}/p\mathbb{Z}\times \{0\}$, $G=\{0\}\times \mathbb{Z}/2\mathbb{Z}$,  $(e_1,f_1,g_1)+(e_2,f_2,g_2)=(e_2,f_2,g_1+g_2)$ and $(e_1,f_1,g_1)\circ (e_2,f_2,g_2)=(e_1+(-1)^{g_1}e_2,f_1+(-1)^{g_1}f_2,e_1+ e_2)$ for all $(e_1,f_1,g_1),(e_2,f_2,g_2)\in B$;
     \item[4)] $(B,\circ)=(\mathbb{Z}/p\mathbb{Z}\times \mathbb{Z}/p\mathbb{Z})\rtimes_{\alpha'} \mathbb{Z}/2\mathbb{Z}$, $E=\mathbb{Z}/p\mathbb{Z}\times \mathbb{Z}/p\mathbb{Z}\times \{0\}$, $G=\{0\}\times \mathbb{Z}/2\mathbb{Z}$,  $(e_1,f_1,g_1)+(e_2,f_2,g_2)=(e_2,f_2,g_1+g_2)$ and $(e_1,f_1,g_1)\circ (e_2,f_2,g_2)=(e_1+(-1)^{g_1}e_2,f_1+f_2,e_1+ e_2)$ for all $(e_1,f_1,g_1),(e_2,f_2,g_2)\in B$;
    \item[5)] $(B,\circ)=(\mathbb{Z}/p^2\mathbb{Z})\rtimes_{\alpha''}\mathbb{Z}/2\mathbb{Z}$, $E=\mathbb{Z}/p^2\mathbb{Z}\times \{0\}$, $G=\{0\}\times \mathbb{Z}/2\mathbb{Z}$,  $(e_1,g_1)+(e_2,g_2)=(e_2,g_1+g_2)$ and $(e_1,g_1)\circ (e_2,g_2)=(e_1+(-1)^{g_1}e_2,g_1+ g_2)$ for all $(g_1,e_1),(g_2,e_2)\in B$;
\end{theor}

\begin{proof}
By \cref{e1p2}, the left semi-brace $B$ is isomorphic to a semidirect product of the trivial left semi-brace $E$ and the trivial skew left brace $G$ of size $2$. By a long but easy calculation, we have that the isomorphism classes of the semidirect product of left semi-braces $E\rtimes_{\alpha} G $ are completely determined by the isomorphism classes of the groups semidirect product $(E\rtimes G,\circ)$.
\end{proof}

\section{Left and right nilpotency of left cancellative left semi-braces}

The goal of this section is to study various types of nilpotency, recently introduced in \cite{catino2022nilpotency}, for the left semi-braces characterized in Section $3$. In particular, we characterize the ones that are right nilpotent and we find sufficient conditions to be right nil or left nil. These results will be applied to partially understand left and right nilpotency of the left semi-braces of size $pq$ and $2p^2$ obtained in Sections $4$ and $5$, many of which will allow us to answer in negative to \cite[Question 1]{catino2022nilpotency}.

\medskip

\noindent Let $B$ be a left cancellative left semi-brace and $\cdotp$ the binary operation on $B$ given by 
$$a\cdotp b:=\lambda_{a}(a^-)+a\circ b+\lambda_b(b^-) $$
for all $a,b\in B$. By \cite[Section 3]{catino2022nilpotency}, we have that $a\cdotp b\in G$, for all $a,b\in B$. Therefore, if $X$ and $Y$ are subsets of $B$, we denote by $X\cdotp Y$ the additive subgroup of $(G,+)$ generated by the elements of the form $a\cdotp b$, where $a\in X$ and $b\in Y$. In the case $X=Y:=B$ one can give two further notions which will be important in this section.

\begin{defin}
Let $B$ be a left cancellative left semi-brace. We define the subset $B^{(n)}$ and $B^n$ by  $B^{(1)}=B^1:=B$ and 
$$B^{(n+1)}:=B^{(n)}\cdotp  B+E\qquad B^{n+1}:=B\cdotp  B^n+E$$
for all $n\in \mathbb{N}$. 
\end{defin}

\noindent In analogy to skew left braces, we say that a left cancellative left semi-brace $B$ is \emph{right} (resp. \emph{left}) \emph{nilpotent} if there exist a natural number $n$ such that $B^{(n)}$ (resp. $B^n$) is equal to the set $E$. As a weaker version of the right nilpotency, we call a left cancellative left semi-brace \emph{right nil} if for every $b\in B$ there exist $n\in \mathbb{N}$ such that $b^{(n)}=0$, where $b^{(1)}:=b$ and $b^{(n+1)}:=b^{(n)}\cdotp b$ for all $b\in B$. As one can expect, the concept of \emph{left nil} left semi-brace is defined in a similar way.\\
As a first result of the section, we characterize left semi-braces of Section $2$ that are right nilpotent.

\begin{prop}\label{rnilp1}
Let $B:=G\rtimes_{\alpha} E$ be a left cancellative left semi-brace as in \cref{teostrutt}. Then, $B$ is right nilpotent if and only if $\alpha_e=id_G$ for all $e\in E$ and $G$ is right nilpotent. 
\end{prop}

\begin{proof}
Suppose that $\alpha_e=id_G$ for all $e\in E$ and $G$ is right nilpotent. Then, $E$ is an ideal of $B$ and therefore by \cite[Corollary 31]{catino2022nilpotency} $B^{(n)}=G^{(n)}+E$ for all $n\in \mathbb{N}$. Therefore, if $G$ is right nilpotent and $m$ is a natural number such that $G^{(m)}=0$, we have that $B^{(m)}=E$.\\
Conversely, if $B$ is right nilpotent we have that $E$ is an ideal of $B$, hence a normal subgroup of $(B,\circ)$. This implies that $\alpha_e=id_G$ for all $e\in E$. Using again the equality $B^{(n)}=G^{(n)}+E$ for all $n\in \mathbb{N}$, the right nilpotency of $G$ also follows.
\end{proof}





\noindent By the previous result, examples of left cancellative left semi-braces that are not right nilpotent occur in abundance among the ones considered in Section $3$. However, a large family of these left semi-braces are right nil, as we can see in the next proposition. At first, we need a technical lemma.


\begin{lemma}\label{nilel}
Let $B:=G\rtimes_{\alpha} E$ be a left cancellative left semi-brace as in \cref{teostrutt}. Then 
$$(g_1,e_1)\cdotp (g_2,e_2)=(-g_1+g_1\circ \alpha_{e_1}(g_2)-g_2,0)$$
for all $g_1,g_2\in G$, $e_1,e_2\in E$.
\end{lemma}

\begin{proof}
It follows by an easy calculation.
\end{proof}

\begin{prop}\label{rnilp2}
Let $B:=G\rtimes_{\alpha} E$ be a left cancellative left semi-brace as in \cref{teostrutt}. Moreover, suppose that $G$ is a trivial skew left brace. Then, $b^{(3)}=0$ for all $b\in B$.
\end{prop}

\begin{proof}
By \cref{nilel}, $(g_1,e_1)^{(2)}=(\alpha_{e_1}(g_1)-g_1,0) $ and $$(g_1,e_1)^{(3)}=(\alpha_{e_1}(g_1)-g_1,0)\cdotp (g_1,e_1)=(\alpha_0(g_1)-g_1,0)=(0,0) $$
for all $g_1\in G$, $e_1\in E$, hence the thesis.
\end{proof}

\noindent As implicitly showed in \cite[Section 5]{catino2022nilpotency}, every right nilpotent left semi-brace is right nil. On the other hand, \cref{rnilp1} and \cref{rnilp2} state that, using the left cancellative left semi-braces characterized in Section $3$, we can construct several examples of finite left cancellative left semi-braces that are right nil but not right nilpotent: as concrete examples, we consider below the left semi-braces obtained in Sections $4$ and $5$. Therefore, a large part of these left semi-braces answers in negative to \cite[Question 1]{catino2022nilpotency}.

\begin{exs}
\begin{itemize}
    \item[1)] Let $B$ be a left cancellative left semi-brace of size $pq$ (where $p,q$ are prime numbers with $p>q$) and with $|E|>1$. Then, if $|E|=p$ or if $(B,\circ)$ is abelian, $B$ is right nilpotent by \cite[Corollary 31]{catino2022nilpotency}. If $|E|=q$ and $(B,\circ)$ is not abelian, then $B$ is isomorphic to the semidirect product of left semi-braces $G\rtimes_{\alpha} E$, therefore by \cref{rnilp1} and \cref{rnilp2} $B$ is a right nil left semi-brace which is not right nilpotent.
    \item[2)] Let $B$ be a left cancellative left semi-brace of size $2p^2$. If $|E|=p^2$ or if $(B,\circ)$ is abelian, by \cite[Corollary 31]{catino2022nilpotency} it is right nilpotent. While, if $|E|=2$, $(B,\circ)$ is not abelian and $G$ is a trivial skew left brace of size $p^2$, by \cref{rnilp1} and \cref{rnilp2} $B$ is a left cancellative left semi-brace right nil but not right nilpotent.

\end{itemize}
\end{exs}

\medskip

\noindent In the rest of this section we focus on left nilpotency. In \cite[Section 5]{catino2022nilpotency} a family of left nilpotent left semi-braces was exhibited. In particular, they provided left nilpotent left semi-braces having nilpotent multiplicative group and with $E$ contained in the centralizer of $G$ or such that is a normal subgroup of $(B,\circ)$. Below, we show sufficient conditions for the left semi-braces of \cref{teostrutt} which ensure left nilpotency.

\begin{prop}\label{leftnilpo}
Let $B:=G\rtimes_{\alpha} E$ be a left cancellative left semi-brace as in \cref{teostrutt} and suppose that $(G,+,\circ)$ is a skew left brace with cyclic additive group and abelian multiplicative group such that the number $s:=|G/Soc(G)|$ is a nilpotent element of the ring $(G,+,*)$ (where $*$ is the usual scalar multiplication). Then, $B$ is a left nilpotent left cancellative semi-brace.
\end{prop}

\begin{proof}
Since $(G,+,\circ)$ has cyclic additive group and abelian multiplicative group, by \cite[Proposition 1]{rump2019classification} there exist $d\in \mathbb{N}\setminus \{1\}$ such that $d=|Soc(G)|$ and $g_1\circ g_2=g_1+g_2+g_1 g_2 d$ for all $g_1,g_2\in G$. By \cite[Theorem 1]{rump2007classification}, we have that $d$ is a nilpotent element of the ring $(G,+,*)$. Now, since $\lambda_{g_1}(g)=(1+dg_1)g$ for all $g,g_1\in G$ and by \cite[Proposition 1]{rump2019classification} for all $e\in E$ there exist $k_e\in \mathbb{N}$ such that $\alpha_e(g)=(1+k_e s) g$ for all $g\in G$, by \cref{nilel}, we have that
$$(g_1,e_1)\cdotp (g_2,e_2)=([(1+dg_1)(1+k_{e_1}s)-1]g_2,0)=((k_{e_1} s+dg_1+dk_{e_1} g_1 s)g_2,0)$$
for all $g_1,g_2\in G$, $e_1,e_2\in E$. Now, if $\{p_1,...,p_z\}$ is the set of prime numbers dividing $|G|$,  by a standard induction and using the nilpotency of $d$ and $s$, one can show that $B^n\subseteq (p_1...p_z)^{n-1} G+E $ for all $n\in \mathbb{N}$ and this clearly implies that $B$ is left nilpotent.
\end{proof}

\noindent \cref{leftnilpo} can be used to construct left nilpotent left cancellative semi-braces that are different from the ones obtained in \cite[Section 5]{catino2022nilpotency}, as we see in the following example.

\begin{ex}
Let $G:=(\mathbb{Z}/p^2\mathbb{Z},+,\circ)$ be the skew left braces given by $a\circ b:=a+b+pab$ for all $a,b\in \mathbb{Z}/p^2\mathbb{Z}$ and $E$ the trivial left semi-brace on $\mathbb{Z}/p\mathbb{Z}$. Let $B$ be the left semi-brace given by $B:=G\rtimes_{\alpha} E$ where $\alpha:E\longrightarrow Aut(G,+,\circ)$ is given by $\alpha_e(g):=(1+pe)g$ for all $e\in E$, $g\in G$. By \cref{leftnilpo} $B$ is a left nilpotent left cancellative semi-brace. On the other hand, $E$ is not contained in the centralizer of $G$ and it is not a normal subgroup of $(B,\circ)$ hence it can not be constructed by \cite[Proposition 44]{catino2022nilpotency}.
\end{ex}

\noindent However, there exist examples of left cancellative left semi-braces provided by \cref{teostrutt} that are not left nilpotent, as we can note by the following example.

\begin{ex}\label{exleft}
Let $B$ be a left cancellative left semi-braces of size $pq$ (where $p,q$ are prime numbers, $p>q$). Then, if $(B,\circ)$ is abelian or $|E|=p$, since $G$ is a trivial skew left brace of size $q$, by \cite[Corollary 14]{catino2022nilpotency} we have that $B$ is left nilpotent. If $(B,\circ)$ is not abelian and $|E|=q$, then $B$ is a semidirect product of left semi-braces as in \cref{teostrutt}. By a standard calculation, we obtain that $B^2=B$, hence $B$ is not left nilpotent.
\end{ex}


\noindent We conclude the paper remarking that in contrast to the right-case, the left semi-braces of size $pq$ are left nil if and only if they are left nilpotent, therefore Question 2 of \cite{catino2022nilpotency}, where the authors ask if a finite left nil is left nilpotent, remains open.



\bibliographystyle{elsart-num-sort}
\bibliography{Bibliography}

\def\cprime{$'$}
\begin{thebibliography}{10}
\expandafter\ifx\csname url\endcsname\relax
  \def\url#1{\texttt{#1}}\fi
\expandafter\ifx\csname urlprefix\endcsname\relax\def\urlprefix{URL }\fi

\bibitem{acri2020skew}
E.~Acri, M.~Bonatto, Skew braces of size $pq$, Comm. Algebra 48~(5) (2020)
  1872--1881.
\newline\urlprefix\url{https://doi.org/10.1080/00927872.2019.1709480}

\bibitem{acri2020retractability}
E.~Acri, R.~Lutowski, L.~Vendramin, Retractability of solutions to the
  {Yang--Baxter} equation and $p$-nilpotency of skew braces, Internat. J.
  Algebra Comput. 30~(01) (2020) 91--115.
\newline\urlprefix\url{https://doi.org/10.1142/S0218196719500656}

\bibitem{bachiller2015classification}
D.~Bachiller, Classification of braces of order $p^3$, Journal Pure Appl. Alg.
  219~(8) (2015) 3568--3603.
\newline\urlprefix\url{https://doi.org/10.1515/forum-2015-0240}

\bibitem{baxter1972partition}
R.~J. Baxter, Partition function of the eight-vertex lattice model, Annals of
  Physics 70~(1) (1972) 193--228.
\newline\urlprefix\url{https://doi.org/10.1016/0003-4916(72)90335-1}

\bibitem{brzezinski2017towards}
T.~Brzezi{\'n}ski, Towards semi-trusses, Rev. Roumaine Math. Pures Appl
  LXIII~(2(2018)).
\newline\urlprefix\url{https://cronfa.swan.ac.uk/Record/cronfa38308}

\bibitem{brzezinski2019trusses}
T.~Brzezi{\'n}ski, Trusses: Between braces and rings, Trans. Amer. Math. Soc.
  372~(6) (2019) 4149--4176.
\newline\urlprefix\url{https://doi.org/10.1090/TRAN%2F7705}

\bibitem{brzezinski2022pre}
T.~Brzezi{\'n}ski, S.~Mereta, B.~Rybo{\l}owicz, From pre-trusses to skew
  braces, Publicacions Matem{\`a}tiques 66~(2) (2022) 683--714.
\newline\urlprefix\url{https://doi.org/10.5565/PUBLMAT6622206}

\bibitem{CaCaSt20x}
M.~Castelli, F.~Catino, P.~Stefanelli, Indecomposable involutive set-theoretic
  solutions of the {Y}ang-{B}axter equation and orthogonal dynamical extensions
  of cycle sets, Mediterr. J. Math. 18, no. 246(2021).
\newline\urlprefix\url{https://doi.org/10.1007/s00009-021-01912-4}

\bibitem{catino2022nilpotency}
F.~Catino, F.~Ced{\'o}, P.~Stefanelli, Nilpotency in left semi-braces, J.
  Algebra 604 (2022) 128--161.
\newline\urlprefix\url{https://doi.org/10.1016/j.jalgebra.2022.04.004}

\bibitem{cacs3}
F.~Catino, I.~Colazzo, P.~Stefanelli, Semi-braces and the {Y}ang-{B}axter
  equation, J. Algebra 483 (2017) 163--187.
\newline\urlprefix\url{https://doi.org/10.1016/j.jalgebra.2017.03.035}

\bibitem{catino2021set}
F.~Catino, I.~Colazzo, P.~Stefanelli, Set-theoretic solutions to the
  {Yang--Baxter} equation and generalized semi-braces, Forum Math. 33~(3)
  (2021) 757--772.
\newline\urlprefix\url{https://doi.org/10.1515/forum-2020-0082}

\bibitem{catino2022set}
F.~Catino, M.~Mazzotta, M.~M. Miccoli, P.~Stefanelli, Set-theoretic solutions
  of the {Yang--Baxter} equation associated to weak braces, Semigroup Forum
  104~(2) (2022) 228--255.
\newline\urlprefix\url{https://doi.org/10.1007/s00233-022-10264-8}

\bibitem{catino2021inverse}
F.~Catino, M.~Mazzotta, P.~Stefanelli, Inverse semi-braces and the
  {Yang-Baxter} equation, J. Algebra 573 (2021) 576--619.
\newline\urlprefix\url{https://doi.org/10.1016/j.jalgebra.2021.01.009}

\bibitem{cedo2014braces}
F.~Ced{\'o}, E.~Jespers, J.~Okni{\'n}ski, Braces and the {Yang-Baxter}
  equation, Comm. Math. Phys. 327~(1) (2014) 101--116.
\newline\urlprefix\url{https://doi.org/10.1007/s00220-014-1935-y}

\bibitem{CeSmVe19}
F.~Ced\'{o}, A.~Smoktunowicz, L.~Vendramin, Skew left braces of nilpotent type,
  Proc. Lond. Math. Soc. (3) 118~(6) (2019) 1367--1392.
\newline\urlprefix\url{https://doi.org/10.1112/plms.12209}

\bibitem{cedo2022new}
F.~Cedó, J.~Okniński, {New simple solutions of the Yang-Baxter equation and
  solutions associated to simple left braces}, J. Algebra 600 (2022) 125--151.
\newline\urlprefix\url{https://doi.org/10.1016/j.jalgebra.2022.02.011}

\bibitem{colazzo2021algebraic}
I.~Colazzo, A.~{Van Antwerpen}, The algebraic structure of left semi-trusses,
  Journal Pure Appl. Algebra 225~(2) (2021) 106467.
\newline\urlprefix\url{https://doi.org/10.1016/j.jpaa.2020.106467}

\bibitem{dietzel2021braces}
C.~Dietzel, Braces of order $p^2q$, J. Algebra Appl. 20~(08) (2021) 2150140.
\newline\urlprefix\url{https://doi.org/10.1142/S0219498821501401}

\bibitem{drinfeld1992some}
V.~G. Drinfel\cprime~d, On some unsolved problems in quantum group theory, in:
  Quantum groups ({L}eningrad, 1990), vol. 1510 of Lecture Notes in Math.,
  Springer, Berlin, 1992, pp. 1--8.
\newline\urlprefix\url{https://doi.org/10.1007/BFb0101175}

\bibitem{etingof1998set}
P.~Etingof, T.~Schedler, A.~Soloviev, Set-theoretical solutions to the {Quantum
  Yang-Baxter} equation, Duke Math. J. 100~(2) (1999) 169--209.
\newline\urlprefix\url{http://doi.org/10.1215/S0012-7094-99-10007-X}

\bibitem{gateva2008matched}
T.~Gateva-Ivanova, S.~Majid, Matched pairs approach to set theoretic solutions
  of the {Yang--Baxter} equation, J. Algebra 319~(4) (2008) 1462--1529.
\newline\urlprefix\url{https://doi.org/10.1016/j.jalgebra.2007.10.035}

\bibitem{gateva1998semigroups}
T.~Gateva-Ivanova, M.~Van~den Bergh, Semigroups of {I-Type}, J. Algebra 206~(1)
  (1998) 97--112.
\newline\urlprefix\url{https://doi.org/10.1006/jabr.1997.7399}

\bibitem{guarnieri2017skew}
L.~Guarnieri, L.~Vendramin, Skew braces and the {Y}ang-{B}axter equation, Math.
  Comp. 86~(307) (2017) 2519--2534.
\newline\urlprefix\url{https://doi.org/10.1090/mcom/3161}

\bibitem{JeVa19}
E.~Jespers, A.~Van~Antwerpen, Left semi-braces and solutions of the
  {Y}ang-{B}axter equation, Forum Math. 31~(1) (2019) 241--263.
\newline\urlprefix\url{https://doi.org/10.1515/forum-2018-0059}

\bibitem{rump2007braces}
W.~Rump, Braces, radical rings, and the quantum {Y}ang-{B}axter equation, J.
  Algebra 307~(1) (2007) 153--170.
\newline\urlprefix\url{https://doi.org/10.1016/j.jalgebra.2006.03.040}

\bibitem{rump2007classification}
W.~Rump, Classification of cyclic braces, Journal Pure Appl. Algebra 209~(3)
  (2007) 671--685.
\newline\urlprefix\url{https://doi.org/10.1016/j.jpaa.2006.07.001}

\bibitem{rump2019classification}
W.~Rump, Classification of cyclic braces, {II}, Trans. Amer. Math. 372~(1)
  (2019) 305--328.
\newline\urlprefix\url{https://doi.org/10.1090/TRAN%2F7569}

\bibitem{rump2019construction}
W.~Rump, Construction of finite braces, Ann. Comb. 23~(2) (2019) 391--416.
\newline\urlprefix\url{https://doi.org/10.1007/s00026-019-00430-1}

\bibitem{rump2019set}
W.~Rump, Set-theoretic solutions to the {Yang--Baxter} equation, skew-braces,
  and related near-rings, J. Algebra Appl. 18~(08) (2019) 1950145.
\newline\urlprefix\url{https://doi.org/10.1142/S0219498819501457}

\bibitem{rump2020one}
W.~Rump, One-generator braces and indecomposable set-theoretic solutions to the
  {Y}ang–{B}axter equation, Proc. Edinb. Math. Soc. (2020) 1–21.
\newline\urlprefix\url{https://doi.org/10.1017/S0013091520000073}

\bibitem{samman1998topics}
M.~Samman, Topics in seminear-ring theory, Annexe Thesis Digitisation Project
  2019 Block 22.

\bibitem{smoktunowicz2018engel}
A.~Smoktunowicz, On {Engel groups, nilpotent groups, rings, braces and the
  Yang-Baxter equation}, Trans. Amer. Math 370~(9) (2018) 6535--6564.
\newline\urlprefix\url{https://doi.org/10.1090/tran/7179}

\bibitem{smock2018skew}
A.~Smoktunowicz, L.~Vendramin, On skew braces (with an appendix by {N. Byott
  and L. Vendramin}), J. Comb. Algebra 2~(1) (2018) 47--86.
\newline\urlprefix\url{https://doi.org/10.4171/JCA/2-1-3}

\bibitem{stefanelli2022affine}
P.~Stefanelli, Semi-affine structures on groups and semi-braces, J. Pure Appl.
  Algebra 227~(No. 107383) (2023) 19 pp.
\newline\urlprefix\url{https://doi.org/10.1016/j.jpaa.2023.107383}

\bibitem{stefanello2022bi}
L.~Stefanello, S.~Trappeniers, On bi-skew braces and brace blocks, J. Pure
  Appl. Algebra 227~(5) (2023) 107295.
\newline\urlprefix\url{https://doi.org/10.1016/j.jpaa.2022.107295}

\bibitem{yang1967}
C.~N. Yang, {Some Exact Results for the Many-Body Problem in one Dimension with
  Repulsive Delta-Function Interaction}, Phys. Rev. Lett. 19 (1967) 1312--1315.
\newline\urlprefix\url{https://link.aps.org/doi/10.1103/PhysRevLett.19.1312}

\end{thebibliography}

\end{document}